\documentclass{amsart}
\usepackage{amssymb,amsthm,mathtools, accents}  
\usepackage{tikz,enumitem,tikz-cd,pgfplots, standalone} 
\usepackage{graphicx}
\usepackage{mathrsfs} 

\usetikzlibrary{arrows}  
\newtheorem{theorem}{Theorem}

\newtheorem{lemma}{Lemma}
\newtheorem{proposition}{Proposition}
\newtheorem{corollary}{Corollary}
\newtheorem*{definition}{Definition}

\theoremstyle{definition}

\newcommand {\D}{\mathcal D}

\newcommand {\lm}{\lambda}

\newcommand{\cal}{Calder\'{o}n }
\usepackage{color}  
\usepackage{hyperref}
\hypersetup{
    colorlinks=true, 
    linktoc=all,     
    linkcolor=blue,  
}

\def\bigskip{\vspace{30pt}}

\begin{document}
\pagestyle{plain} 
\title{Adiabatic Limit of Calder\'{o}n Projector on Manifold with Cylindrical End}
\author{Kunal Sharma}
\date{}

\begin{abstract}
For a Reimannian manifold with a cylindrical end, consider a Dirac-type operator that is asymptotically product type with the generalized Atiyah---Patodi---Singer boundary condition on any finite portion of the cylinder. In the present work we consider the problem of constructing the \cal projector in this setting and studying the adiabatic limit of it along the cylindrical end. As a consequence, we extend a result of Nicolaescu \cite{N} on adiabatic limits of Cauchy data spaces. The proof leverages resolvent and its estimates in the framework of the $b$-calculus needed for the construction of the \cal projector corresponding to our Dirac-type operator. 
\end{abstract}

\maketitle

\tableofcontents


\fontsize{11}{20pt} \selectfont  



\section{Introduction}

The seminal work of Atiyah---Patodi---Singer(\cite{APS}) demonstrated the need to have global boundary conditions for the study of Dirac-type operators. On the other hand, the study of elliptic operators on manifold with boundary with local boundary conditions had been studied before for example by Lions and Magenes(\cite{LM}), Hormander(\cite{Ho}) and Boutet de Monvel(\cite{B}). In particular, in \cite{B}, Boutet de Monvel identified conditions on symbols of pseudo-differential operators to study Fredholmness followed by an ingenious calculation of the index using K-theory. To set up a calculus of pseudo-differential operators analogous to \cite{B} and study boundary value problems for Dirac-type operators in a systematic way, a projection operator on the $L^2$ space of the boundary is needed. Such an operator was constructed by Seeley in \cite{S1}, extending the work of \cal, and independently by Hormander in \cite{Ho2}. We will henceforth refer to this projection operator as the \cal projector. For details about constructing such projectors and reformulations of boundary conditions in Atiyah---Patodi---Singer's work(\cite{APS}) along with `generalized APS' boundary conditions in the context of \cal projectors, we will rely and refer the reader to the wonderful manuscript by Booß-Bavnbek and Wojciechowski \cite{BB}. In the setting of compact manifold with boundary, the construction of \cal and treatment of boundary value problems for elliptic operators under more relaxed conditions can be found in the work of Booß-Bavnbek et al. (\cite{BARB}) and Bär et al. (\cite{BLZ} and \cite{BARBH}) respectively. \par

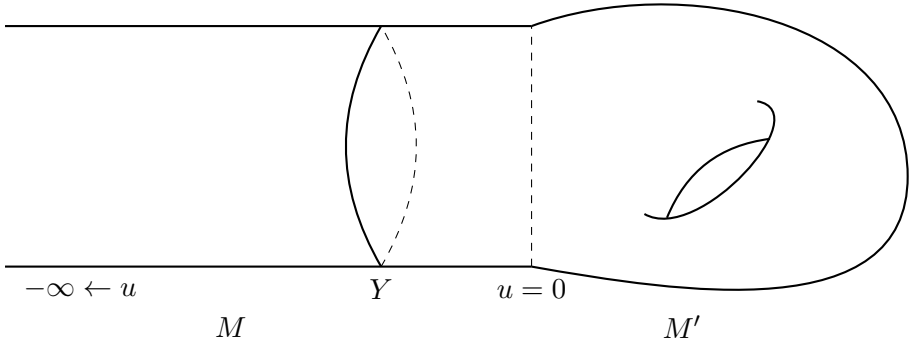
\begin{figure}[h!]
  \begin{center}
    \begin{tikzpicture}
       \draw [thick] (1,2) 
       to [out=20,in=90] (6,0) 
       to [out=-90,in=-10] (1,-1.2);    
      \draw [thick] (2.5,-0.5) to [out=-30, in=-10] (4,1);
      \draw [thick] (2.8,-0.55) to [bend left] (4.17, 0.5);
      \draw [thick](1,2) -- (-6,2);
      \draw [thick] (1, -1.2) -- (-6,-1.2);
      \draw (-1,-1.5) node {$Y$};
      \draw [thick] (-1,-1.2) to [bend left] (-1,2);
      \draw [dashed] (-1,-1.2) to [bend right] (-1,2);
      \draw [dashed] (1,-1.2) -- (1,2);
      \draw (1,-1.5) node {$u=0$};
      \draw (-5,-1.5) node {$- \infty \leftarrow u$};
      \draw (3,-2) node {$M'$};
      \draw (-3,-2) node {$M$};
    \end{tikzpicture}
    \caption{Manifold with cylindrical end}
  \end{center}
\end{figure}

The problem we address in this work is that of studying \cal  projectors for Dirac-type operators on manifolds with an asymptotically cylindrical end (see Figure 1). If $M$ is such a manifold then it admits a decomposition $M =  ((-\infty,0]_u \times Y) \cup M'$ where $M'$ is compact with boundary $Y$ and the Riemannian metric $g$ on $(-\infty,0]_u \times Y$ takes the form \[ g = du^2 + h_u, \text{  }  h_u \sim h_Y + e^uh_1 + e^{2u}h_2 + \cdots  \text{ as u} \to -\infty, \]  where $h_Y$ is a metric on $Y$. Therefore, $g \to du^2 + h_Y$ as $u \to -\infty$. A Dirac-type operator acting between sections of a Hermitian bundle $E$ over $M$ that's non-product type is given as $\D = G(\partial_u + \D_0(u))$ over the cylindrical end $(-\infty, 0]_u \times Y$, where $G$ is a unitary map and $\D_0(u)$ is a Dirac-type operator acting on $C^{\infty}(Y,E_0)$, where $E_0$ is $E$ restricted to $\{u=0\}$.  Here $\D_0(u)$ has an expansion similar to the one of $g$ as $u \to -\infty$ and given as, \[  \D_0(u) \sim D_0 + e^u D_1 + e^{2u} D_2 + \cdots, \text{ where } D_j \in \text{Diff}^1(Y, E_0). \] 
Let $M_r = ([r,0] \times Y) \cup M' $ for which we consider a parameter dependent \cal projector $\mathcal C (r)$, corresponding to the Dirac operator $\D$, which by definition is an orthogonal projection onto the Cauchy data space at $u =r$,  \[ \mathcal H(r) = \overline{ \{ \phi|_{u =r}| \phi \in C^{\infty}(M_r,E),\D \phi =0\} } \subset L^2(Y,E_0). \] 
The question the present work addresses is to compute the adiabatic limit of the \cal projector i.e. to compute $\lim_{r \to -\infty} \mathcal C(r).$ More precisely, assume the `generalized APS' boundary condition $\Pi_{APS}= \Pi_{>} + \Pi_{sc}$, where $\Pi_>$ is the positive spectral projection of $D_0$ and $\Pi_{sc}$ is projection on the \textit{scattering Lagrangian}, $L \subset ker D_0$. As $M$ is asymptotically cylindrical we expect that $\mathcal C(r) \to \Pi_{APS}$ as $r \to -\infty$. This conjecture is also motivated by work of Nicolaescu(\cite{N}) wherein the asymptotic limit of the Cauchy data space was studied for a manifold with cylindrical end. There are some differences however between our work and \cite{N}. Firstly, our metric $g$ is asymptotically cylindrical and not pure product-type. Explicitly, it looks like $\D = G(\partial_u + \D_0(u))$ where $\D_0(u)$ is smooth in $u$. The Dirac-type operator assumed in \cite{N} on the other hand is `cylindrical' i.e. on the cylindrical part of the manifold it looks like $\D = G(\partial_u + \D_0)$ where  $\D_0$ is independent of cylindrical variable $u$. A much more crucial difference is in the approach: Recall $M(r) = ( [r, 0] \times Y) \cup M'$ and let $\mathcal H(r)$ be defined as above. In \cite{N} the objective was to study $\lim_{r \to -\infty} \mathcal H(r)$, whereas we instead consider the limit of the \cal projector $\lim_{r \to -\infty} \mathcal C(r)$. Now $\mathcal H(r) \subset L^2(Y,E_0)$ and on the cylinder $\D = G(\partial_u + \D_0)$, so formally speaking we have $\mathcal H(r) = \mathcal H(0)e^{-\D_0 r}$, which is a family of Lagrangian spaces. Taking the limit requires studying dynamics on the space of infinite dimensional Lagrangian Grassmanians. Our approach of dealing with the \cal projector, we believe, is more direct and robust for we study convergence of operators rather than the converges of spaces. As a consequence of continuity of the \cal projector the result on the adiabatic limit of the Cauchy data spaces (Corollary 4.11 in \cite{N}) follows. Moreover, instead of working with $C^{\infty}$ spaces we work with \textit{symbol spaces}. Roughly speaking, this space sits between the category of smooth functions and spaces of polyhomogenous conormal distributions as defined in \cite{M}. In addition, this approach is applicable to more singular spaces. The main tool we use to deal with the analysis is Melrose's $b$-calculus. In this setting the adiabatic limit is equivalent to the limit of kernels of pseudodifferential operators in the product space $M^2$. \par

This work is organized as follows: The first section sets up the necessary notations and introduces core definitions related to the `$b$-calculus', mainly that of $b$-pseudodifferential operators and blow-up. Then a construction of the \cal projector is provided based on the resolvent estimates. Due to the latter, our construction doesn't rely on an invertible double for $\D$. Finally, we prove that the adiabatic limit of the \cal projector is the generalized APS projector for both the product case and asymptotically cylindrical case and we also provide the associated rates of convergence. 
 
 \subsection{Acknowledgements}
I would like to thank Paul Loya for several helpful discussions and in general for introducing me to the area of boundary value problems for elliptic operators.


\section{Background and Notation}

\subsection{Blow-up}
In order to deal with the singularities of operators on $M$ it's convenient to instead work with a `compactified' manifold obtained by change of variables over the cylindrical part, \[ (-\infty, 0] \ni u \to x \coloneqq e^u \in (0,1]. \] Consequently, $x = 0$ represents the `boundary at infinity'. This allows us to leverage the $b$-calculus of Melrose(\cite{M}), aspects of which we briefly recall now.\par 

Let $X$ denote the compactified manifold and consider the product manifold $X^2$ with the corner $\partial X \times \partial X \eqqcolon C$ represented by $\{ x = x' = 0 \}$ where $x,x'$ are the boundary defining functions of $X$. Here, $y$ represents the coordinates on the boundary $Y = \partial M'$ so $(x,y)$ are coordinates near the boundary of Y.  As a convention we always use the unprimed coordinates, respectively primed coordinates, as the corresponding coordinates on the left, respectively right, boundary functions of $X^2$.  If $P \in \Psi^m(X)$ is a pseudodifferential operator on $X$  then the (cornormal) singularities of its kernel lie along the diagonal, $\bigtriangleup = \{(x,y) = (x',y') \}$ and it's imperative then to study the behavior as we approach the corner $C$. This can be achieved by `blowing-up' the corner $C$ where we define the \textit{blow-up} as $X^2$ with the introduction of polar coordinates at the origin $C$ and denote the resulting manifold as $X^2_b$.

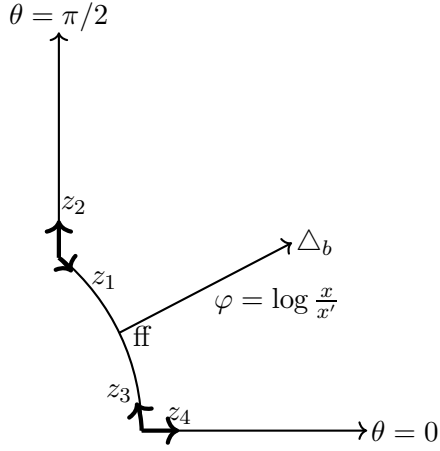
\begin{figure}
\begin{tikzpicture}
\draw[thick] (3,0) arc (0:50:3cm) ;
\draw[thick, ->] (1.9,2.3) -- (1.9,5.3) ;
\draw[thick, ->] (3,0) -- (6,0) ;
\node at (1.9, 5.5) {$\theta = \pi/2$} ;
\node at (6.5,0) {$\theta = 0$} ;
\draw[ultra thick, ->] (1.9,2.3) -- (1.9,2.8) ;
\draw[ultra thick, ->] (3,0) -- (3.5,0) ;
\draw[ultra thick, ->] (1.9,2.3) arc (50:35:1cm) ;
\node at (2.1, 3.0) {$z_2$} ;
\node at (2.5, 2) {$z_1$} ;
\draw[ultra thick, ->] (3,0) arc (0:22:1cm);
\node at (3.5, 0.2) {$z_4$} ;
\node at (2.7, 0.5) {$z_3$} ;
\draw[thick, ->] (2.7,1.30) -- (5.0,2.5);
\node at (5.3,2.5) {$\bigtriangleup_b$};
\node at (4.8, 1.7) {$\varphi = \log \frac{x}{x'}$};
\node at (3.0, 1.26) {ff};
\end{tikzpicture} 
\caption{Blow-up of $X^2$}
\end{figure}

 A much more general definition however of the blow-up in the framework of $b$-calculus with more details can be found in \cite{Me}(Sec. 4.2). If $(r, \theta) $ are the polar coordinates at $\{ x = x' = 0 \}$ then the `front face', ff,  in $X^2_b$ is \[ \text{ff} = \{ r = 0 \} \times [0, \pi/2] \times Y^2.  \]  Around the left and right boundaries of $X^2_b$, $z_i$ are the chosen projective coordinates and near left boundary for instance $(z_1, z_2) = (\frac{x}{x'}, x').$ In Figure 2, $\varphi$ denotes the logarithmic projective coordinate defined for regions away from left and right boundaries. Then, setting $\omega = (\varphi, y - y')$ defines coordinates normal to the $b$-diagonal where $\bigtriangleup_b$ is given as $\varphi = 0$.   \par

\subsection{b-pseudodifferential operators}
\label{subsection:bpsdo}
As a result of the compactification of $M$, the metric over the cylindrical part of $X$ takes the form \[ g =  \frac{dx^2}{x^2} +  h_x, h_x \sim h_Y + xh_1 + x^2h_2 + \cdots \] Moreover, in regard to the $\partial_u \to x\partial_x$ transformation, the class of $b$-differential operators can be defined explicitly (we suppress the dependence on the Hermitian bundle $E$ here for notational simplicity) 
\begin{definition}($b$-differential operators)
Let $\text{Diff}^m_b(X)$ denote the class of $b$-differential operator of order $m$ on $X$. Then $P \in \text{Diff}^m_b(X)$ if in a coordinate patch near the boundary the following expression holds,
 \[ P = \sum_{\vert \alpha \vert + \vert \beta \vert \leq m}a_{\alpha \beta}(x,y)(x\partial_x)^{\alpha}\partial^\beta_y\] where $a_{\alpha \beta}(x,y)$ are smooth functions supported in that patch. 
 \end{definition}
 
Below we define ``symbol spaces'' that are function classes broader than smooth functions. These spaces are defined according to the decay rates of the functions as they approach the boundary of the manifold. They can be compared with conormal distributions as defined in Sec. 18.2 in \cite{Ho} or more directly with Sec 5.10 in \cite{Me}. 
\begin{definition}(Symbols of order $0$)
Let $X$ be as above then $f: \mathring X \to \mathbb C$ is a symbol of order zero if $\forall P \in \text{Diff}_b^*(X)$, $Pf \in L^{\infty}$. Denote the class of such functions as $S^0(X)$.
\end{definition}
Now let $\rho \in C^{\infty}(X)$ be a boundary defining function for $\partial X$ which by definition satisfies $\rho \geq 0, d\rho \neq 0$ on $\partial X = \{ \rho =0 \}$. Then the above definition can be generalized to define the following space
 \[ S^{\alpha}(X) \coloneqq \rho^{\alpha}S^0(X),   \alpha \in \mathbb R. \] 
 Next we define $S^{0,\alpha}$ space to allow control over the behavior of kernels near the boundaries and the front face in the blow-up space $X^2_b$. 
\begin{definition} 
Denote by $S^{0,0}(X^2_b)$  the space of all $u$ such that:
\begin{enumerate}[label = \roman*)]
\item $u \in S^0(X^2_b)$ is continuous up to the left and right boundaries of $X^2_b$
\item $ \exists $ $ 0 < \epsilon < 1 $ such that to the front face, \[ u(r, \theta, y,y') = u_0(\theta,y,y') + r^{\epsilon}u_1(r,\theta,y,y'), \] 
where $u_0 \in S^0([0, \pi/2]\times Y^2)$ and is continuous up to $\theta = 0$ and $\theta = \pi/2$, and $u_1 \in S^0([0,1) \times [0, \pi/2] \times Y^2)$ 
\end{enumerate}
\end{definition}
\begin{definition}
For $\alpha >0$, define the space $S^{0,\alpha}(X^2_b)$ as containing functions $u$ smooth in the interior of $X$ such that for some $\delta > 0$,
\begin{enumerate}[label = \roman*)]
\item $u$ vanishes to the order $\alpha + \delta$ on the left and right boundaries of $X^2_b$ (where $\theta = 0$ and $\theta = \pi/2$ respectively)
\item To the front face, $u$ has an expansion of the form 
\[ u(r, \theta, y, y') = u_0(\theta, y, y') + ru_1(\theta, y, y') + \cdots + r^k u_k(\theta, y, y') + r^{\alpha} u_{k+1}(r, \theta, y, y'),  \text{ } k = \lfloor \alpha \rfloor,  \] 
where \[ u_i \in \theta^{\alpha + \delta}(\pi/2 - \theta)^{\alpha + \delta}S^0([0,\pi/2] \times Y^2), \text{ } i = 0, ... , k   \] 
while \[ u_{k+1} \in  \theta^{\alpha + \delta}(\pi/2 - \theta)^{\alpha + \delta}S^0([0,1) \times [0,\pi/2] \times Y^2)  \]
\end{enumerate} 
\end{definition}

By virtue of the Schwartz kernel theorem, pseudodifferential operators are in one-to-one correspondence with their kernels(conormal distributions with singularity at diagonal), a fact that will be frequently exploited below. 
\begin{definition}
Let $ \Psi^{-\infty,\alpha}_b(X) $ denote the class of $b$-pseudodifferential operators with kernels in $S^{0,\alpha}(X^2_b)*^b\Omega_R$ where $^b\Omega_x = \Omega(^bT^*_xX)$ is the b-density bundle and $^b\Omega_R$ is lifted to the right and in local coordinates can be written as $\frac{dx'}{x'}dy'$.
\end{definition}
    
Generalizing this, a broader class of $m$th-order $b$-pseudodifferential operators can be defined. 
\begin{definition}
	We define $\Psi^{m,\alpha}_b(X)$ as the space of operators $P$ that are standard pseudodifferential operators in $\accentset{\circ}{X^2}$ and whose kernels $K_P$ satisfy the following:
	\begin{enumerate}[label = (\roman*)]
	   \item If $\phi \in C^{\infty}_0(X^2_b \setminus \bigtriangleup_b)$, $\phi K_P \in S^{0,\alpha}(X^2_b)$ 
	   \item If $\phi \in C^{\infty}_0(X^2_b)$ is supported in a neighborhood $\mathcal U \cong [0,1)_r \times \mathbb R_y^{n-1}\times \mathbb R^n_z$ off $ff \cap \bigtriangleup_b$, and $\mathcal U \cap \bigtriangleup_b \cong \{z=0\}$ then 
	       \[ \phi K_P = \int e^{i \langle z, \xi \rangle}a(r,y, \xi)d\xi\otimes \mu_R \]
	       where $a(r,y,\xi) \in S^{m,\alpha}([0,1)_r \times \mathbb R^{n-1}_y; \mathbb R^n)$ which means
	       \[a(r,y, \xi) = \sum_{l=0}^{k} r^la_l(y,\xi) + r^{\alpha}a_{k+1}(r,y,\xi), \text{ } k = \lfloor \alpha \rfloor \] 
	       where $a_l$ is a standard symbol of order m and $a_{k+1}$ is a symbol of order zero in $r$ and a standard symbol of order $m$ in $(y,\xi)$. 
	\end{enumerate}
\end{definition}

This allows us to define the ``small'' calculus as $\Psi^m_b(X) \coloneqq \bigcap\limits_{\alpha > 0}\Psi^{m,\alpha}_b(X)$
Moreover, the following characterizes the operators in the small calculus, 
\begin{lemma}
	We have $P\in \Psi^m_b(X)$ if and only if $K_P$ vanishes to infinite order \text{} (in Taylor series) to the left and right boundaries of $X^2_b$ while when supported on a patch near the ff $\cap \bigtriangleup_b$, 
	 $K_P = \int e^{i \langle z, \xi \rangle}a(r,y,\xi)d\xi, \text{ } a\in S^m([0,1) \times \mathbb R^{n-1}; \mathbb R^n)  $. 
\end{lemma}


\section{Construction of the \cal projector}

Recall that the Dirac-type operator $\D \in \text{Diff}^1_b(X)$ on the cylindrical end of $X$ has the following form,  \[ \D(x) = G(x\partial_x + \D_0(x))  \text{ where } \D_0(x) \in C^{\infty}([0,1]_x, \text{Diff}^1(Y,E_0)) \] Moreover, the assumption of smoothness on $\D_0$ implies the expansion \[ \D_0(x) \sim D_0 + xD_1 + x^2D_2 + \cdots  \text{ as } x \to 0 \] where $D_i \in \text{Diff}^1(Y,E_0)$.  \par In order to study the adiabatic limit of the \cal projector, we impose the following condition on $\D$,


\begin{definition}{(Nonresonance condition)}
Let $\mathcal H$ be the Cauchy data space for $\D$ and let $L_{<} \subset L^2(E_0)$ be the space spanned by eigenvectors corresponding to the negative eigenvalues of $D_0$. Then $\D$ is said to satisfy the \textit{nonresonance condition} if $\mathcal H \cap L_< = 0.$ 
\end{definition}
It was shown in \cite{N} that under this condition on $\D$, \[ \lim_{x \to 0} \mathcal H(x) = range(\Pi_{APS}), \text{    }\Pi_{APS} = \Pi_> + \Pi_{sc}  ,\] where $\Pi_>$ is the  projector on the positive eigenvalues of $D_0$ and $\Pi_{sc}$ is the `scattering Lagrangian'. 
Note that $\D$ is nonresonant if and only if the null space of $\D$ has no elements in $L^2$. This can be seen by expanding $f(u,y) \in C^{\infty}((-\infty,0]_u \times Y, E)$ in terms of the orthogonal basis of $D_0$ as $ f(u,y) =  \sum_{\lambda }e^{-u\lambda}f_{\lambda}(0)\phi_{\lambda}(y)$. The $L^2$ solutions then correspond to $\lm < 0$. For a detailed proof, see \cite{K}

Consider now the restriction to $ X_r \coloneqq ([r,1]_x \times Y) \cup M' $. Bundles and operators defined above for $X$ retain their properties on $X_r$ by virtue of restriction. As $X_r$ is compact, following \cite{S1} and \cite{BB}, the \cal projector for $\D(r)$, $ \mathcal P_{\D(r)}$ could be defined as 
\[  \mathcal P_{\D(r)}\phi(y) \coloneqq \lim_{r \downarrow x} \gamma_r {\widetilde{\D(r)}}^{-1}\delta(r-x)\otimes G\phi(y), \: \: \: \:  \forall  \phi(y) \in C^{\infty}(Y_r)  \]  
where $\widetilde{\D(r)}$ is the extension of $\D(r)$  defined on the invertible double. However our construction below differs from the usual construction as it doesn't rely on the double manifold for the exact inverse of $\D$. Rather, $\D$ being a $b$-differential operator, we recall the following construction from Melrose (\cite{M}).  

\subsection{Parametrix of $\D(r)$}
Based on the construction of resolvent, ${(\D^2 + \lambda)}^{-1}, \lambda \in \mathbb C \backslash \mathbb R$ in Melrose (Ch 6, \cite{M}) for an exact $b$-metric $\D$, as in present case, a resolvent for $\D$ could be derived. Explicitly, it is given as  \[(\D - \lm)^{-1} = R(\lm) + \frac{1}{\lm}\Pi_0,  \] 
where $\Pi_0$ is the projection on the $L^2$ null space and $R(\lm) \in \Psi^{-1, \lm}_b(X, E_0)$. \par 
Note that it follows from the nonresonance condition that $\Pi_0$ disappears. Moreover, from the meromorphic extension of the resolvent, taking $\lm\to0$ with the analytic continuation $\Im \lm > 0$, we get 
\begin{equation}\label{eq: r0}
R(0) = Q + K, \text{ where } Q \in \Psi^{-1}_b(X) \text{ and } K \in  \Psi^{-\infty,0}_b(X). 
\end{equation}
Here, $Q$ belongs to the `small' calculus and $K$ belongs to `big' calculus, here $\Psi^{-\infty}$. Equipped with the desired inverse of $\D$, a projector could be constructed using the definition mentioned above. 

\subsection{Adiabatic limit of the projector} 
The aim here is to study the projector corresponding to $R(0)$ on the Cauchy data spaces $\mathcal H(r) \subset L^2(Y_r)$ as a function of $r \in [0,1)$. Following the definition of the \cal projector and using the inverse of $\D(r)$ from previous section we define a (non unique) projection operator as 
\[ \mathcal P_{\D(r)} = \gamma_r Q \delta(x-r)\otimes G + \gamma_r K \delta(x-r)\otimes G. \] 
Denote the first term in the sum as $\mathcal P_Q(r)$ and the second one as $\mathcal P_K(r)$. Then we have the following for $\mathcal P_Q(r)$ where the proof closely follows Loya and Park, \cite{LP}
\begin{theorem}
$\mathcal P_Q(r): C^{\infty}(Y) \to C^{\infty}(Y)$ is well defined and \[\mathcal P_Q(r) \in C^{\infty}([0,1)_r, \Psi^0(Y)) .\]
\end{theorem}
\begin{proof}
Consider $\mathcal P_Q(r)$ for a fixed small $r$ (as the `boundary' being at $r =0$) as an operator on $X_r$. Then on $X^2_b$, it is sufficient to consider $\mathcal P_Q$ in a coordinate patch $\mathcal U$ supported near front face, ff, such that $\mathcal U \cap \bigtriangleup_b \neq \emptyset$. Locally then, by choosing a projective coordinate system $\{(\frac{x}{x'}, x, y) \} \in \mathbb R_{(\frac{x}{x'})} \times [0,1)_x \times R^{n-1}_y, \mathcal P_Q(r)$ can be described as 
\begin{equation}\label{eq: pq}
\begin{split}
\mathcal P_Q(r)\phi(y) &= \lim_{x \downarrow r} \gamma_x Q {(\delta(r)\otimes G\phi)} \\
                                   &= (2\pi)^{-2n}\lim_{x \downarrow r} \int_{\mathbb R^{2n}} \left(\frac{x}{x'}\right)^{i\tau}e^{i\langle y, \eta \rangle} a(x,y,\tau,\eta) \delta(x'-r)\otimes G \phi(y')\frac{dx'}{x'}dy'd\tau d\eta \\
                                   &= (2\pi)^{-n}\lim_{x \downarrow r} \int_{\mathbb R^{n-1}}e^{i\langle y, \eta \rangle} \int_{\mathbb R} \left(\frac{x}{r}\right)^{i\tau}a(x,y,\tau,\eta) G\hat \phi(\eta)d\tau d\eta \\
                                   &= (2\pi)^{-n+1}\int_{\mathbb R^{n-1}}e^{i\langle y, \eta \rangle} \tilde a(r,y,\eta) G\hat \phi(\eta) d\eta 
\end{split}
\end{equation} 
where \[ \tilde a(r,y,\eta) =  (2\pi)^{-1}\lim_{x \downarrow r} \int_{\mathbb R} \left(\frac{x}{r}\right)^{i\tau}a(x,y,\tau,\eta) d\tau .\] In the second step above, kernel of $Q$ is expanded allowing the Fourier transform $\hat \phi(\eta)$ in $y' \to \eta$ in the third step. Action of $\delta(x' -r)$ on $(x/x')$ can be justified by approximating the Dirac measure in the usual way. For example, by choosing a sequence $u_{\epsilon}(t) \in C^{\infty}_0(\mathbb R)$, supported near $t=r$ and integrating to 1. Recall that $\mathcal P_Q(r)$ belongs to `small' calculus as mentioned following \eqref{eq: r0}. This implies that $a(x,y,\tau,\eta)$ is smooth in $x$ for $x \in [0,1)$

\par Now consider $\tilde a(r,y,\eta)$ and make the change of variables $x \to xr$. Then,
\begin{equation*}
\tilde a(r,y,\eta) =  \lim_{x \downarrow 1} \int_{\mathbb R} e^{i\langle ln x, \tau \rangle} a(rx,y,\tau,\eta) d\tau
\end{equation*}
As $Q \in \Psi^{-1}_b(X)$ it has rational symbol i.e. \[ a(rx,y,\tau,\eta) \sim \sum_{j=1}^{\infty} a_j(rx,y,,\tau,\eta) \]  where 
\[ a_j = \frac{p_j(rx,y,\tau,\eta)}{(\tau^2 + \| \eta \|_{g_Y}^2)^j} = \frac{p_j(rx,y,\tau,\eta)}{(\tau + i \| \eta \|_{g_Y})^j(\tau - i \| \eta \|_{g_Y})^j}, \]  
with $a_j\in C^{\infty}(T^*(X))$ of order $-j$ and $p_j$ is a polynomial of order $j$ in $(\tau,\eta)$. Moreover, the $b$-principal symbol of $\D$ is $a_1^{-1}(xr,y,\tau,\eta) = i\tau + \sigma(D_0)(\eta)$, where $\sigma(D_0)$ is the principal symbol of $D_0$ and by virtue of $D_0$ being Dirac-type, $\sigma^2(D_0)(\eta) = \| \eta \|_{g_Y}^2, \eta \in T^*(Y)$. Note that if $\Im \tau > 0 $ then $e^{i\langle ln x, \tau \rangle} \in O(\tau^{-N})$ for any $N$. This suggests that the integral can be replaced by a contour integral by shifting $\mathbb R$ over the roots of rational function $a(rx,y,\tau,\eta)$ lying in upper half of $\mathbb C$. Therefore let $\tau$ be moved to $\Im \tau = \infty$. Having the necessary holomorphic extension of $a_j$ for each $j$, application of Cauchy's theorem expresses $\tilde a(r,y,\eta)$ as a sum of residues. By abuse of notation, let $Q = \sum_{j=1}^{\infty}OP(a_j)$, where $OP(k)$ is the pseudo-differential operator with symbol $k$. This allows writing $\tilde a$ as a sum,
\begin{equation}
\begin{split}
\tilde a(r,y,\eta)  &=  (2\pi)^{-1}\lim_{x \downarrow 1} \int_{\mathbb R} e^{i\langle ln x, \tau \rangle} \sum_{j=1}^{\infty} a_j(rx,y,\tau,\eta) d\tau \\
                         &= (2\pi)^{-1}\lim_{x \downarrow 1} \sum_{j=1}^{\infty} \int_{\mathbb R}e^{i\langle ln x, \tau \rangle} \frac{p_j(rx,y,\tau,\eta)}{(\tau + i \| \eta \|_{g_Y})^j(\tau - i \| \eta \|_{g_Y})^j} d\tau \\
                         &= \lim_{x \downarrow 1} \sum_{j=1}^{\infty}\frac{i}{(j-1)!}\left.{{\left(\frac{d}{d\tau}\right)}^{j-1}}\right\vert_{\tau = i\| \eta \|_{g_Y}}\frac{p_j(rx,y,\tau,\eta)}{(\tau + i \| \eta \|_{g_Y})^j}\\
                         &= \frac{i\| \eta \|_{g_Y} + \sigma(\D_0)(r,y,\eta)}{2\| \eta \|_{g_Y}}   + \\
                         & \sum_{j=2}^{\infty}\frac{i}{(j-1)!}\left.{{\left(\frac{d}{d\tau}\right)}^{j-1}}\right\vert_{\tau = i\| \eta \|_{g_Y}}\frac{p_j(r,y,\tau,\eta)}{(\tau + i \| \eta \|_{g_Y})^j}  \\
                         &\eqqcolon  \tilde {a_0}(r,y,\eta) + \sum_{l=1}^{\infty} \tilde {a_l}(r,y,\eta)                              
\end{split} 
\end{equation}
It can be checked that $\tilde a_l(r,y,\eta)$ are symbols homogeneous of order $l$ in $\eta$. Therefore $\tilde a(r,y,\eta)$ is a symbol of order $0$. 
\par Now substituting $\tilde a_l(r,y,\eta)$ in \eqref{eq: pq} we get $\mathcal P_Q(r)$ as a well-defined pseudodifferential operator at the boundary $\{ r \} \times Y$ i.e. $\mathcal P_Q(r) \in \Psi^0(Y)$. Moreover, as 
\begin{equation}
\tilde a_0(r,y,\eta) = \frac{1}{2}\Big(Id +  \frac{\sigma(\D_0)(r,y,\eta)}{\| \eta \|_{g_Y}}\Big)
\end{equation}
the principal symbol of $\mathcal P_Q(r)$, $\tilde a_0(r,y,\eta)$, is an orthogonal projection on the space corresponding to positive eigenvalues of \[ \sigma(\D_0)(r,y,\eta): C^{\infty}(Y_r; E_r) \to C^{\infty}(Y_r; E_r). \] \par 
Finally, collecting the above and noting that smoothness of $a(x,y,\tau, \eta)$ in $x$ implying smoothness of $\tilde a_l(r,y,\eta)$ in $r$ proves that $\mathcal P_Q(r) \in C^{\infty}([0,1)_r, \Psi^0(Y)) $

\end{proof}

\begin{corollary}
The limit $\lim_{r \to 0^+} \mathcal P_Q(r)$ exists. 
\end{corollary}

A similar computation as above can be applied to the second term in \eqref{eq: r0}, $K$, to study its Poisson operator on $X_r$. 

\begin{proposition}
There is an $\epsilon > 0$ such that $\mathcal P_K(r) =  \mathcal P_{K_0}(0) + r^{\epsilon}\mathcal P_{K_1}(r)$, where $\mathcal P_{K_0} \in \Psi^{-\infty}(Y)$ and $\mathcal P_{K_1}(r) \in S^0([0,1)_r, \Psi^{-\infty}(Y))$. 
\end{proposition}
\begin{proof}
As previously, it is sufficient to consider $\mathcal P_K(r)$ for small r. So assume that $\mathcal P_K(r)$ is supported in a coordinate patch $\mathcal U$ near front face, ff, such that $\mathcal U \cap \bigtriangleup_b \neq \emptyset$. In the patch $\mathcal U \subset X^2_b$ we can use the coordinate system be given by $ \{(\frac{x + x'}{2}), x/x', y)\}$. Denoting the Schwartz kernel of $K$ as $\tilde K$, we have
\begin{equation}
\begin{split}
\mathcal P_K(r)\phi(y) &= \gamma_r K \delta(x-r)\otimes G\phi(y) \\
                          &= \gamma_r \int_{[0,1) \times \mathbb R^{n -1}} \tilde K\Big(\frac{x + x'}{2}, x/x', y, y' \Big) \delta(x' -r) \otimes G\phi(y')\frac{dx}{x'}'dy' \\
                          &= \gamma_r \int_{\mathbb R^{n-1}} \tilde K\Big(\frac{x + r}{2}, x/r, y, y' \Big) G\phi(y')dy' \\
                          &= \int_{\mathbb R^{n-1}} \tilde K\Big(r, 1, y, y' \Big) G\phi(y')dy' 
\end{split}
\end{equation}
Note that $K$ is a \textit{b}-smoothing operator, so $\gamma_r $ is the usual restriction without needing to take limits. Now as $K \in \Psi_b^{-\infty,0}(X)$, $\tilde K$ has an expansion to the front face, $\{r=0\}$, given as \[ \tilde K(r, y, y') =  \tilde K_0(y,y') + r^{\epsilon}\tilde K_1(r,y,y'), \text{ for some } \epsilon > 0. \] Here the kernels $\tilde K_i \in S^0(X^2_b)$. This implies 
\begin{equation}
\mathcal P_K(r)\phi = \mathcal P_{K_0}\phi + r^{\epsilon} \mathcal P_{K_1}(r)\phi
\end{equation}
where \[ \mathcal P_{K_0}(r)\phi(y) = \int_{\mathbb R^{n-1}} \tilde K_0\Big(y, y' \Big) G\phi(y')dy'  \in  \Psi^{-\infty}(Y) \] and  \[ \mathcal P_{K_1}(r) = \int_{\mathbb R^{n-1}} \tilde K_1\Big(r, y, y' \Big) G\phi(y')dy'    \in S^0([0,1)_r, \Psi^{-\infty}(Y)) \]
\end{proof}

As a consequence of the above results on $\mathcal P_Q(r)$ and $\mathcal P_K(r)$ it follows that $\forall r \in [0,1)$ there is an $\epsilon > 0$ such that the following holds for the projector corresponding to $R(0)$:
\begin{equation}\label{eq: Pr}
\mathcal P_{\D(r)} = \mathcal P_{Q}(r) + r^{\epsilon}\mathcal P_K(r),
\end{equation}
where $\mathcal P_{Q}(r) \in C^{\infty}([0,1), \Psi^0(Y))$ and $\mathcal P_{K}(r) \in S^0([0,1), \Psi^{-\infty}(Y))$. \par

Note that having $\mathcal P_{R}(r)$, the \cal projector for the adjoint $\D^* = (-x\partial_x + \D_0)G$,  $\mathcal P^*(\D)$ could be constructed in a similar manner as above.  

\subsection{Consequences} 
Although $\mathcal P_{D(r)}$ is a projector it however is not an orthogonal projector, something which is needed to relate to $\Pi_{APS}$. But it can be turned into an orthogonal operator in the standard way (see \cite{BB})
\begin{lemma}
Given $\mathcal P_{\D(r)}$, an orthogonal projection is given as \[  \mathcal C(r) \coloneqq \mathcal P_{\D(r)}\mathcal P^*_{\D (r)}{[\mathcal P_{\D(r)}\mathcal P^*_{\D(r)} + (Id - \mathcal P_{\D(r)})(Id - \mathcal P^*_{\D(r)}) ]}^{-1} \in \Psi^{0}(Y) \]  Moreover, for some $\epsilon > 0, \mathcal C(r) = \mathcal C_1(r) + r^{\epsilon}\mathcal C_2(r)$, $\mathcal C_1(r) \in C^{\infty}([0,1), \Psi^0(Y))$ and \\$\mathcal C_2(r) \in S^0([0,1), \Psi^{-\infty}(Y))$. 
\end{lemma}
\begin{proof}
As both $\mathcal P_{\D(r)}$ and $\mathcal P^*_{\D}$ are projection operators on $L^2$, it's clear that $\mathcal C(r)$ is a projection on $Ran (\mathcal P_{\D(r)}), \mathcal C(r) \in \Psi^0(Y)$, and $\mathcal C^*(r) = C(r)$ (see \cite{LP}). The expansion of $\mathcal C(r)$ essentially follows from \eqref{eq: Pr} and a similar expansion of $\mathcal P^*_{\D (r)}$. It can be verified that \[ \mathcal A \coloneqq C^{\infty}([0,1), \Psi^0(Y)) + r^{\epsilon}S^0([0,1), \Psi^{-\infty}(Y)) \] is an Algebra that is closed under adjoints and inverses. It follows that \[  T \coloneqq {[\mathcal P_{\D(r)}\mathcal P^*_{\D(r)} + (Id - \mathcal P_{\D(r)})(Id - \mathcal P^*_{\D(r)}) ]}^{-1} \in \mathcal A \] and therefore $\mathcal P_{\D(r)}\mathcal P^*_{\D (r)}T \in \mathcal A$.

\end{proof}

As $\mathcal C_1(r)$ is smooth in $r$, the above can be rewritten as $\mathcal C(r) = {\mathcal C_1}(0) + r^{\epsilon} {\mathcal C_3}(r)$ where ${ \mathcal C_3}(r)\in S^0([0,1)_r, \Psi^0(Y))$. The non-resonance condition now implies that \[ \lim_{r \to 0} \mathcal C(r) = {\mathcal C_1}(0) = \Pi_{APS}.  \] \par 
Consider now the product case, so let $\D = G(x\partial_x + \D_0)$ where $\D_0$ is independent of $x$. As $\mathcal C(r)$ and $\Pi_{APS}$ have the same principal symbol, $\mathcal C(r) - \Pi_{APS} \in \Psi^{-1}(Y)$.  However, this result can be sharpened to have an error given by smoothing operator. We restate the result from Grubb(\cite{Gr})

\begin{lemma}
If $\D = G(x\partial_x + \D_0)$ then for any $r \in (0,1)$, $\mathcal C(r) - \Pi_{APS} \in \Psi^{-\infty}(Y)$.
\end{lemma}
The idea of proof is to glue the parametrices on the cylinder and interior using smooth compactly supported functions as in \cite{APS} and establishing Sobolev estimates. 

As a consequence of above and collecting the results we have our main theorem describing the \textit{adiabatic limit} of the \cal projector
\begin{theorem}
Assuming the non-resonance condition, there is an $\epsilon > 0$ such that as $r \downarrow 0 $ we have \[ C(r) - \Pi_{APS} = r^{\epsilon}f(r), where \]
\begin{enumerate}[label = (\roman*)]
\item $f(r) \in S^0([0,1)_r, \Psi^0(Y))$ in the case of an asymptotically cylindrical end. 
\item $f(r) \in S^0([0,1)_r, \Psi^{-\infty}(Y)$ in the product case, $\D = G(x\partial_x + \D_0)$.
\end{enumerate}
\end{theorem}


\clearpage

\phantomsection

\vspace{4mm}


\end{document}